\documentclass{article}

\usepackage{amssymb, amsfonts, amsbsy, latexsym, epsfig, color}
\usepackage{amsmath}
\usepackage{amsthm}
\usepackage{amsmath,amssymb}
\usepackage{url}
\usepackage{fullpage}
\usepackage{paralist}

\newcommand{\Z}{\mathbb{Z}}
\newcommand{\N}{\mathbb{N}}

\newcommand{\R}{\mathbb{R}}

\newcommand{\intRtwo}{\int \limits_{\R^2}}

\newcommand{\wv}{\omega}
\newcommand{\sh}{\psi}

\newcommand{\diag}[2]{\begin{pmatrix}#1 & 0\\ 0& #2\\\end{pmatrix}}

\newtheorem{theorem}{Theorem}[section]
\newtheorem{definition}[theorem]{Definition}

\newtheorem{lemma}[theorem]{Lemma}

\DeclareMathOperator{\suppp}{supp \,}

\begin{document}
 
\title{Shearlet approximation of functions with discontinuous derivatives}
\author{Philipp Petersen}
\date{}
\maketitle
\begin{abstract}
We demonstrate that shearlet systems yield superior $N$-term approximation rates compared with wavelet systems of functions whose first or higher order derivatives are smooth away from smooth discontinuity curves. We will also provide an improved estimate for the decay of shearlet coefficients that intersect a discontinuity curve non-tangentially.
\end{abstract}


\noindent
\section{Introduction}

In applied harmonic analysis one important field of study is the design of suitable systems $(\varphi_n)_{n\in \N} \subseteq L_2(\R^2)$ that can efficiently represent functions $f\in L_2(\R^2)$ in the sense that $f = \sum_{n\in \N} c_n(f) \varphi_n$ for a sequence $(c_n(f))_{n\in \N}$. 
One particularly desirable feature of such a system is that by using only few elements it already yields decent approximations of functions taken from some subset of $L_2(\R^2)$. The quality of approximation within such a system is customarily measured in terms of the error of the best $N$- term approximation. 

The \emph{best $N$-term approximation} of $f$ is given by
\begin{equation*}
\sigma_N(f) = \inf \limits_{\substack{E_N \subset \N, |E_N| = N,\\ \tilde{f}_N = \sum \limits_{n \in E_N} c_n \varphi_n}} \|f- \tilde{f}_N \|_{L_2(\R^2)}^2.
\end{equation*}
If for some class of functions $\Theta \subseteq L_2(\R^2)$ one has $\sigma_N(f) = O(g(N))$ for $N\to \infty$ for all $f\in \Theta$ and some function $g$, we will call $g$ a \emph{best $N$-term approximation rate for $\Theta$}.

One particular type system $(\varphi_n)_{n\in \N}$ used efficient representation of functions are wavelet systems, see \cite{TenLectures} and the introduction in Subsection \ref{sec:Wavelets}. Wavelet systems have established themselves as a standard tool for image analysis. The wavelet construction is based on \emph{isotropic} scaling of a generator function, which implies that the supports of all elements $\varphi_n$ have the same aspect ratio. Due to this construction wavelets perform poorly when representing functions that contain anisotropic components. For instance a function could have a discontinuity along a smooth curve. In this event wavelets yield non-optimal approximation rates as we will recall in Subsection \ref{sec:CartoonApprox}. 

To overcome this shortcoming shearlets were introduced in \cite{GKL2006}. These systems constitute an excellent tool for the approximation or representation of natural images due to the fact that they provide almost optimal best $N$-term approximation rates for functions that have discontinuities along smooth curves. We will give more details on shearlet systems in Subsection \ref{sec:Wavelets}. 

Naturally one can pose the question about approximation rates of wavelets and shearlets for other classes of functions which exhibit anisotropic structures. Let, for example, $u$ be the solution of the operator equation 
$$ L u = f,$$
where the data $f$ exhibits a discontinuity along a smooth curve and $L$ is an elliptic differential operator. The approximation rates of shearlets for such functions have not been studied yet, despite the fact, that good approximation properties for such functions are crucial for the design of optimal adaptive solvers for elliptic partial differential equations. In fact, it is well established, see \cite{Ste2003, DahFR2007, cohen2001adaptive}, that using wavelet systems as ansatz functions, certain partial differential equations (PDEs) can be solved with an adaptive strategy with a computational complexity depending on the best $N$-term approximation rate of the solution of the PDE admitted by the chosen system. In view of applications of shearlet systems for the discretization of PDEs it is thus important to know whether the best $N$-term approximation rate of shearlets outperforms that of wavelets also for classes of functions that contain solutions of elliptic partial differential equations. In fact, if 
this were not the case, it would not be worthwhile to study adaptive frame methods with shearlets instead of wavelets.

Motivated by these considerations we examine functions with first or higher order derivatives which are smooth away from smooth discontinuity curves. We will first consider functions which have a $l$-th order derivative which is cartoon-like in the classical sense, i.e., it is twice continuously differentiable apart from a twice continuously differentiable discontinuity curve. 
Cartoon-like functions are a standard model in image processing to model natural images. They where first analyzed in \cite{DCartoonLikeImages2001}. By now analyzing the approximation rates for the class of cartoon-like functions was established as the standard benchmark problem to determine the quality of a representation system, such as shearlets \cite{KLcmptShearSparse2011, GuoL2007}, but also curvelets, \cite{CurveletsIntro}, and more general systems, \cite{PMolecules}.

We will observe, that shearlets yield a superior best $N$-term approximation rate of such functions of $O(N^{-{l-2}})$ for $N\to \infty$ (Theorem \ref{thm:1}), when compared to wavelets whose approximation rate cannot be faster than $O(N^{-{l-1}})$ for $N\to \infty$ (Theorem \ref{thm:0}). However, we will see, that not the discontinuity but the the regularity assumption on the smooth parts in the cartoon-like model limits the approximation rates. 
Hence we introduce a smoother cartoon-like model in Subsection \ref{sec:DiffCartoon} and we will observe in Theorem \ref{thm:2} that the approximation rates by shearlets improve drastically to $O(N^{-2l - 7/2})$ for $N\to \infty$, while those of wavelet systems remain unchanged at $O(N^{-{l-1}})$ for $N\to \infty$. 
In order to obtain the approximation rate of Theorem \ref{thm:2} we prove an improved estimate for shearlet coefficients that do not intersect a discontinuity curve tangentially in Lemma \ref{lem:decayTheorem}.

Contemplating the results of this paper in the context of adaptive frame methods discussed above, it is fair to say that the fact that shearlets yield improved approximation rates over wavelets for functions with cartoon-like derivatives can serve as a justification for the study of adaptive frame methods with shearlet systems instead of wavelets systems. We expect that these results will trigger the developement of new adaptive frame methods based on shearlet systems.

\section{Preliminaries}
\subsection{Notation}
For $0< p\leq \infty$, we will denote by $L_p$ the usual Lebesque spaces of $p-$times integrable functions and by $\ell_p$ we denote the corresponding sequence spaces. For $l\in \N$ we denote by $W^l_p$ the spaces of $l-$times weakly differentiable functions with derivatives in $L_p$ and by $C^l$ the spaces of $l-$times continuously differentiable functions. For a set $B \subset \R^2$ we use $\partial B$ to describe its boundary and $\chi_B$ is its characteristic function. For two functions $h,g$ we write $h \lesssim g$ is there exists a constant $C$ such that $h(x) \leq C g(x)$ for all $x$ in the domain of $h,g$. 
\subsection{Wavelets and shearlets}\label{sec:Wavelets}

2D Wavelet systems are constructed from integer translations and dilations of a generator function $\wv\in L_2(\R^2)$, i.e. for $j \in \N$ and $m\in \Z^2$
$$
\wv_{j,m}(x): = 2^{j}\wv(2^j x - m)\text{ for all } x\in \R^2.
$$
It can be shown, that there exist $\wv^0,\wv^1, \wv^2, \wv^3\in L_2(\R^2)$ such that the \emph{wavelet system}
$$
\mathcal{W}: = \{\wv_{0,m}^0,\ m\in \Z^2\} \cup \{\wv^i_{j,m},\ j\geq 0,\ m\in \Z^2,\ i= 1,2,3\}
$$
yields an orthonormal basis for $L_2(\R^2)$, see \cite{TenLectures}. We see that the scaling of the wavelet elements is \emph{isotropic}. In the construction of shearlets this isotropic scaling is replaced by parabolic scaling in combination with a shearing matrix. In particular, shearlet systems are constructed using the following two matrices:
\begin{equation*}
S_k = \begin{pmatrix} 1 & k \\ 0 & 1 \end{pmatrix}, \text{ and } A_j = \begin{pmatrix} 2^j & 0 \\ 0 & 2^{\frac{j}{2}} \end{pmatrix}, \quad \text{ where } k,j \in \Z.
\end{equation*}
We use \emph{cone-adapted shearlet systems} which are defined as follows.
\begin{definition}\cite{KLcmptShearSparse2011}
Let $\phi, \psi \in L_2(\R^2)$, $c= [c_1,c_2]^T \in \R^2$ with $c_1,c_2>0$. Then the \emph{(cone-adapted) shearlet
system} is defined by
\begin{equation*}
\mathcal{SH}(\phi, \psi, \tilde{\psi}, c) = \Phi(\phi, c_1) \cup \Psi(\psi, c) \cup \tilde{\Psi}(\tilde{\psi}, c),
\end{equation*}
where
\begin{eqnarray*}
\Phi(\phi, c_1) &:=& \left \{ \sh_{0,0,m,0} = \phi(\cdot - c_1 m) :m\in \Z^2 \right\},\\
\Psi(\psi, c) &:=& \left\{ \sh_{j,k,m,1} = 2^{\frac{3j}{4}}\psi(S_k A_{j}\cdot -  M_c m ): j\in \N_0, |k| \leq   2^{\left\lceil\frac{j}{2}\right\rceil}, m\in \Z^2 \right\},\\
\tilde{\Psi}(\tilde{\psi}, c) &:=& \left\{ \sh_{j,k,m,-1} = 2^{\frac{3j}{4}}\tilde{\psi}(S_k^T \tilde{A}_{j}\cdot -  M_{\tilde{c}} m ): j\in \N_0, |k| \leq
2^{\left\lceil\frac{j}{2}\right\rceil}, m\in \Z^2 \right\},
\end{eqnarray*}
with $\tilde{\psi}(x_1,x_2) = \psi(x_2,x_1)$, 
$$
M_c:= \diag{ c_1}{ c_2 }, M_{\tilde{c}} = \diag{ c_2}{c_1 }, \text{ and }\tilde{A}_{j} = \diag{2^{\frac{j}{2}}}{2^{j}}.
$$
\end{definition}

For cone-adapted shearlet systems we will employ the index set $\Lambda := \{(j,k,m,\iota): |\iota| j \geq j\geq 0, |k|\leq |\iota| 2^{\frac{j}{2}}, m\in \Z^2, \iota = \{1,0,-1\} \}$. Let furthermore $d_{[-\pi/2, \pi/2]}$ denote the following metric on the torus $\mathbb{T} = [-\pi/2, \pi/2]$: 
$$d_{[-\pi/2, \pi/2]}(a, b) := \min (|a-b|, \pi + \min\{a,b\} - \max\{a,b\}).$$
A shearlet is called \emph{separable}, if $\psi(x_1,x_2) = \psi^1(x_1)\phi^1(x_2)$ for two functions $\psi^1, \phi^1 \in L_2(\R)$.

Under certain assumptions, see \cite{KGLConstrCmptShear2012} cone-adapted shearlet systems can form a \emph{frame}, i.e. there exist $0< c_1 \leq c_2 <\infty$ such that 
\begin{align*}
c_1 \|f\|_2 \leq \sum_{\lambda \in \Lambda} |\langle f, \sh_\lambda \rangle|^2 \leq c_2 \|f\|_2.
\end{align*}
This implies, that there exists a \emph{dual frame} $(\sh_\lambda^d)_{\lambda \in \Lambda}$ such that 
\begin{align}\label{eq:frameProp}
T^d: \ell^2 \to L_2(\R^2), \qquad (c_\lambda)_{\lambda\in \Lambda} \mapsto T^d((c_\lambda)_{\lambda\in \Lambda}):=\sum_{\lambda \in \Lambda} c_\lambda \sh_\lambda^d,
\end{align}
is a bounded operator and $f = T^d((\langle f, \sh_\lambda \rangle)_{\lambda \in \Lambda})$ for all $f\in L_2$, see \cite{Chr}.

\subsection{Approximation of cartoon-like functions}\label{sec:CartoonApprox}

The improvement of shearlets over wavelets becomes evident, when one considers their approximation rates for classes of functions that model natural images, so-called cartoon-like functions.
\begin{definition}
The set of \emph{cartoon-like functions} is given by 
\begin{align*}
\mathcal{E}^2(v) := \{f \in L_2(\R^2): f = g_1 + \chi_B g_2, \text{ where } g_1,g_2 \in C^2, \suppp g_1,g_2 \subset (0,1)^2\\
\text{ and } B \subset (0,1)^2 \text{ with } \partial B \in C^{2} \text{ and }\partial B \text{ has curvature bounded by $v$} \}.
\end{align*}
We call $\partial B$ the \emph{discontinuity curve of} $f$. 
\end{definition}
It can be shown that wavelets only achieve a best $N$-term approximation rate of $O(N^{-1})$ for the class of cartoon-like functions, \cite{CurveletsIntro}. On the other hand in  \cite{KLcmptShearSparse2011} and \cite{GuoL2007} it was shown that cone-adapted shearlet systems achieve a best $N$-term approximation rate of $O(N^{-2}\log(N)^{3})$, which - up to the $\log$ factor - is the optimal approximation rate that any system can achieve, see \cite{DCartoonLikeImages2001}.

\section{Results}

We are concerned with approximation rates of functions that have some first or higher order derivative which is cartoon-like. 
Let $l \in \N$, and 
\begin{equation*}
\mathcal{E}^{l,2}(v)  : = \left\{u\in W^{l}_2(\R^2) \ : \ \frac{\partial^\alpha}{\partial x^\alpha} u = u^\alpha \text{ and }u^\alpha \in \mathcal{E}^2(v) \text{ for all }  |\alpha| = l\right\}.
\end{equation*}
Note that in the definition of cartoon-like functions $g_2 = 0$ is possible, so that not all $l-$th derivatives of $u\in \mathcal{E}^{l,2}(v)$ need to have a discontinuity curve.
We will show, that the best $N$-term approximation rate for $\mathcal{E}^{l,2}(v)$ by wavelet systems is bounded from below by $g(N) = N^{-(l+1)}$ and that there exists a best $N$-term approximation rate by shearlets of the order of $g(N) = N^{-(l+2)}$.

\subsection{Wavelet approximation rates}
In order to analyze the approximation properties of wavelets we measure the smoothness of the function to be approximated in a Besov scale. We use the following definition of a Besov space, which can be found in \cite{Triebel1}:

For $h >0$ and $f\in L_2(\R^2)$ we define $\Delta_h f = f(x+h) - f(x)$. For $s = k + r$, where $k\in \mathbb N$ and $0< r \leq 1$ and $1\leq p,q<\infty$ the Besov space $B_{p,q}^s$ is defined by
\begin{align*}
B_{p,q}^s :&= \{f\in W_{p}^{k}: \ \|f\|_{B_{p,q}^s} <\infty \}\\
\|f\|_{B_{p,q}^s}: &= \|f\|_{W_{p}^{k}} + \sum_{|\beta| = k}\left( \int_{\mathbb R^2}  \frac{\|\Delta_h^2 D^\beta f\|_{L_p(\mathbb R^2)}^q}{|h|^{n+q r }} dh\right)^{\frac{1}{q}}.
\end{align*}

Many wavelet systems on a domain on $\R^2$ admit a characterization of Besov spaces $B_{q,q}^s$ by non-linear approximation rates in the following sense:
\begin{align} \label{eq:BesovChar}
\sum_{n = 1}^\infty [n^{s/2}  \sigma_n(u)^{1/2}]^r \frac{1}{n} <\infty \Leftrightarrow u \in B^{s}_{q,q},
\end{align}
for $1/q = s/2 + 1/2$ and $r\geq s$, see for instance \cite{DeJP1992, Coh2000}.

Using \eqref{eq:BesovChar} we can now find an upper bound for the approximation rate of wavelets for functions in $\mathcal{E}^{m,2}_p(v)$.
\begin{theorem}\label{thm:0}
Let $l \in \N$, $0<p<\infty$, $u\in W^{l}_p$ such that for some $|\alpha| = l$ we have 
\begin{equation*}
\frac{\partial^\alpha}{\partial x_i^\alpha} u = \chi_{D},
\end{equation*}
where $D$ is a bounded subset of $\R^2$ with smooth boundary curve $\partial D$ that has bounded curvature. Let $\mathcal{W}$ be a wavelet system such that \eqref{eq:BesovChar} holds. 
Then for all $\epsilon>0$ we have that $\sigma_N(u) \not \in  O(N^{-(l+1)-\epsilon})$.
\end{theorem}
\begin{proof}
By a simple computation we obtain that $\chi_{D} \not \in B^{1}_{q,q}$ for any $0<q <\infty$. Consequently, $u \not \in B^{l+1}_{q,q}$ for any $0<q<\infty$. If 
$\sigma_N(u) \lesssim N^{-(l+1)-\epsilon}$ for $N\to \infty$ for some $\epsilon<0$ it would follow with $s = l+1$, $1/q = (s + 1)/2$ applied to \eqref{eq:BesovChar} that $u \in B^{l+1}_{q,q}$, which is a contradiction.
\end{proof}

\subsection{Shearlet approximation rates}
Now let us prove that for $0\leq l \in \N$, shearlet systems obtain a higher approximation rate than wavelets for functions in $ \mathcal{E}^{l,2}(v)$. 

\begin{theorem} \label{thm:1}
Let $l\geq 0$, and $0\leq v<\infty$, and $u\in\mathcal{E}^{l,2}(v)$ and let $\mathcal{SH}(\phi,\sh, \tilde{\sh}, c) = (\sh_{\lambda})_{\lambda \in \Lambda}$ be a shearlet system with shearlet generator $\sh  = \left(\frac{\partial}{\partial x_1}\right)^l\theta$ such that $\theta\in L_2(\R^2)$ has compact support, is separable and
\begin{compactenum}[(i)]
 \item $|\hat{\theta}(\xi)| \lesssim \min(1, |\xi_1|^\tau) \min(1,|\xi_1|^{-\nu}) \min(1, |\xi_2|^{-\nu})$ \text{ and }
 \item $|\frac{\partial}{\partial \xi_2} \hat{\theta}(\xi) | \leq |h(\xi_1) (1+\frac{\xi_2}{\xi_1})^{-\nu}|$,
\end{compactenum}
where $\tau>5, \nu\geq 4, h\in L_1(\R)$. 
Further assume that $(\sh_{\lambda})_{\lambda \in \Lambda}$ forms a frame. Then 
\begin{align}\label{eq:shearEstimate}
\|u-\sum_{\lambda \in E_N} \langle u, \sh_\lambda \rangle \sh_\lambda^d \|_{L_2(\R^2)}^2 \lesssim N^{-(l+2)}\log(N)^3,
\end{align}
where $E_N \subset \Lambda$ contains the indices of the $N$ largest coefficients $|\langle u, \sh_\lambda \rangle|$ and $\sh_\lambda^d$ is the canonical dual frame element of $\sh_\lambda$. The $\log$ term in \eqref{eq:shearEstimate} can be dropped if $l\geq 1$.
\end{theorem}
\begin{proof}

For $l = 0$ the result is Theorem 1.3 of \cite{KLcmptShearSparse2011}. In the sequel we assume $l\geq 1$. Let $\gamma_1$ be the discontinuity curve of $u_1 := \left(\frac{\partial}{\partial x_1}\right)^l u$ and $\gamma_{-1}$ be the discontinuity curve of $u_{-1}= \left(\frac{\partial}{\partial x_2}\right)^l u$. We denote for $t\in \gamma_i$ the outer normal at $t$ of $\gamma_i$ by $n^i(t)$, $i= {1,-1}$. Now we decompose our index set $\Lambda$. First of all we denote by 
$$
\Lambda^0 = \{ (j,k,m,\iota) \in \Lambda: \iota = 0\}.
$$
By the compact support of $u$ and $\sh$ we have that $\#\{\lambda \in \Lambda^0: \langle u, \sh_{j,k,m,\iota}\rangle \neq 0 \} < \infty$. Furthermore we denote 
$$
\Lambda^1 : = \{(j,k,m,\iota)\in \Lambda \setminus \Lambda^0: \suppp \sh_{j,k,m,\iota} \cap \bigcup_{\iota = \{-1,1\}} \gamma_\iota = \emptyset \}.
$$ 
We denote by $\Lambda^{2,a}$ the indices in $\Lambda^{2}: = \Lambda \setminus (\Lambda^0 \cup \Lambda^1)$, such that for $(j,2^{j/2}s,m,\iota) \in \Lambda$ there exists some $t\in \suppp \sh_{j,2^{j/2}s,m,\iota} \cap \gamma_\iota$ such that we have  $s^{\iota}\in (\tan \rho_0-3\cdot2^{-j/2}, \tan \rho_0+3\cdot2^{-j/2})$ and $n^{\iota}(t)= \pm(\cos \rho_0, \sin \rho_0)$ for some $\rho_0 \in (-\pi/2,\pi/2)$. Additionally, there are non-aligned shearlet elements that intersect the discontinuity curve, we will call these indices $\Lambda^{2,b}$. We have $\Lambda^{2,b} = \Lambda \setminus (\Lambda^{0}\cup \Lambda^1 \cup \Lambda^{2,a})$.


Now let us observe the sizes of the index sets for fixed scale $j$ and we denote these index sets by $\Lambda^1_j,\Lambda^{2,a}_j, \Lambda^{2,b}_j$. 
We have that $\#\{(j,k,m,\iota) \in \Lambda^1_j:\langle u, \sh_{j,k,m,\iota}\rangle \neq 0\} \lesssim 2^{2j}$. 
Observe that, due to their bounded curvature, $\gamma_1$ and $\gamma_{-1}$ have finite length.
Hence we observe that
\begin{equation*}
|\Lambda^{2,a}_j| \lesssim 2^{\frac{j}{2}}, \quad |\Lambda^{2,b}_j| \lesssim 2^{\frac{3}{2} j}. 
\end{equation*}

\textbf{Part 1: }$\Lambda^0$:

There are only finitely many indices in $\Lambda^0$, i.e. we certainly have
\begin{equation*}
\sum_{n\geq N} c(u)^*_n  \lesssim N^{-(l+2)}, 
\end{equation*}
where $c(u)^*$ denotes a non-increasing rearrangement of $(|\langle u, \sh_{j,k,m,\iota} \rangle|^2)_{(j,k,m,\iota)\in \Lambda^0}$.

\textbf{Part 2: }$\suppp \sh_{j,k,m,\iota} \cap \bigcup_{\iota = \{-1,1\}} \gamma_\iota = \emptyset:$

In this case the scalar products $\langle u, \sh_{j,k,m,\iota} \rangle$ decay as if the function was in $C^{l+2}$. In other words, by invoking Proposition 2.1 of \cite{KLcmptShearSparse2011} (which was only stated for $l = 0$, but the extension is straightforward) this means that
\begin{equation}\label{eq:bottleneck}
\sum_{n\geq N} c(u)^*_n  \lesssim N^{-(l+2)}, 
\end{equation}
where $c(u)^*$ denotes a non-increasing rearrangement of $(|\langle u, \sh_{j,k,m,\iota} \rangle|^2)_{(j,k,m,\iota)\in \Lambda^1}$.

\textbf{Part 3: }$\suppp \sh_{j,k,m,\iota} \cap \bigcup_{\iota = \{-1,1\}} \gamma_\iota \neq \emptyset:$

Using partial integration and the compact support of $\theta$ one obtains 
\begin{align}
|\langle u, \sh_{j,k,m,\iota} \rangle |\lesssim 2^{-l j} |\langle u_\iota, \theta_{j,k,m,\iota} \rangle|. \label{eq:decayForDeriv}
\end{align}
Invoking Proposition 2.2. in \cite{KLcmptShearSparse2011}, which are the standard estimates for shearlets i.e. $(l = 0)$, in combination with \eqref{eq:decayForDeriv} yields that for $(j,k,m,\iota)\in \Lambda^{2,a}$
\begin{equation*}
|\langle u, \sh_{j,k,m,\iota}\rangle| \lesssim 2^{-(\frac{3}{4}+l)j}
\end{equation*}
and for $(j,k,m,\iota) \in \Lambda^{2,b}$ we have
\begin{equation*}
|\langle u, \sh_{j,k,m,\iota}\rangle| \lesssim 2^{-(\frac{9}{4}+l)j}.
\end{equation*}
Using $p = 2/(l+3)$ and the sizes of $\Lambda^{2,a}_j, \Lambda^{2,b}_j$ we can compute, that 
\begin{align}
\sum_{\Lambda^{2}} |\langle u, \sh_{j,k,m,\iota}\rangle|^p \leq& \sum_{\Lambda^{2,a}} |\langle u, \sh_{j,k,m,\iota}\rangle|^p +  \sum_{\Lambda^{2,b}} |\langle u, \sh_{j,k,m,\iota}\rangle|^p \nonumber \\
\lesssim &\sum_{j\in \N} 2^{\frac{3}{2} j} 2^{-p(\frac{9}{4}+l)j} +  \sum_{j\in \N} 2^{\frac{j}{2}} 2^{-p(\frac{3}{4}+l)j} <\infty.\label{eq:decofCoeffs}
\end{align}
Stechkin's lemma, see e.g. \cite{Dev1998}, states that 
\begin{equation}
(\sum_{n\geq N} |d_n|^2)^{\frac{1}{2}} \lesssim N^{-s}, \label{eq:Stechkin}
\end{equation}
if $(d_n)_n$ is a monotonically decreasing sequence such that $(d_n)_n \in \ell_p$ for $s = 1/p - 1/2$.
Using Stechkin's lemma, we obtain that $\sum_{n\geq N} c(u)^*_n  \lesssim N^{-(l+2)}$, where $c(u)^*$ denotes a non-increasing rearrangement of $(|\langle u, \sh_{j,k,m,\iota} \rangle|^2)_{(j,k,m,\iota)\in \Lambda^2}$. 
Using the estimates from Part 1,2, and 3 and invoking the dual frame characterization \eqref{eq:frameProp} yields
\begin{align*}
\|u-\sum_{\lambda \in E_N} \langle u, \sh_\lambda \rangle \sh_\lambda^d \|_{L_2(\R^2)}^2 \lesssim \sum_{n\geq N} c(u)^*_n  \lesssim N^{-(l+2)},
 \end{align*}
where $c(u)^*_n$ denotes a non-increasing rearrangement of $(|\langle u, \sh_{\lambda} \rangle|^2)_{\lambda \in \Lambda}$.
\end{proof}

\begin{figure}[htb]
 \includegraphics[width = \textwidth]{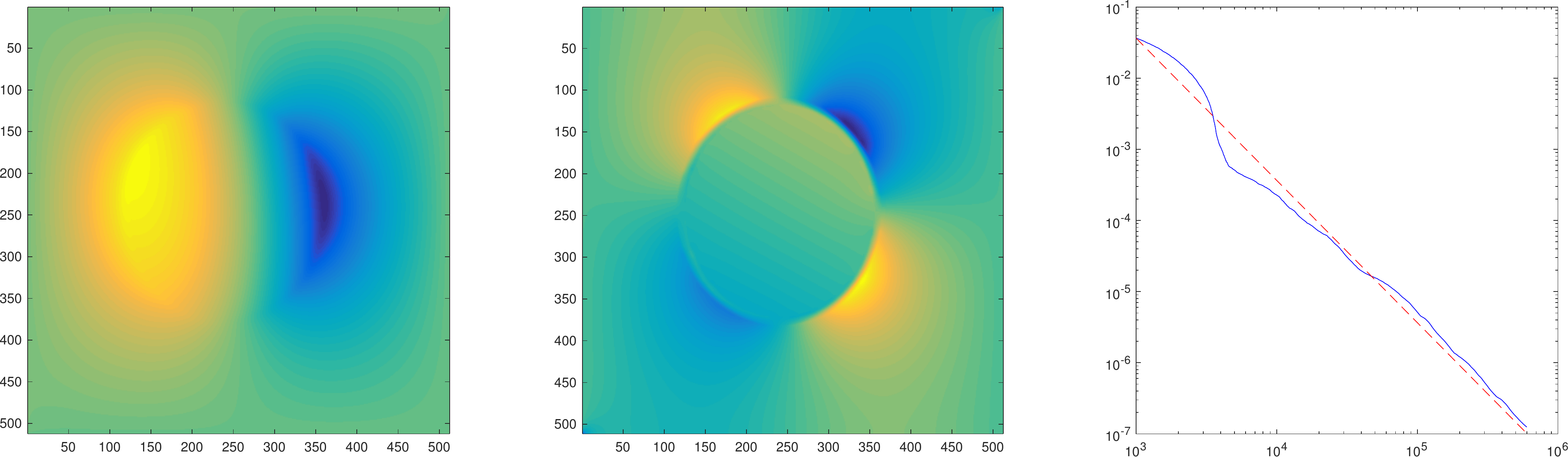}
 \put(-70, -2){\small $n$}
 \put(-155, 60){\rotatebox{90}{\small $(c(u)^*_n)^{\frac{1}{2}}$}}
 \caption{\textbf{Left:} Function $u$ with cartoon-like derivatives. \textbf{Middle:} Cartoon-like first derivative of $u$ in vertical direction. \textbf{Right:} Plot on a logarithmic scale of the predicted decay of $O(n^{-2})$ for $n\to \infty$ due to \eqref{eq:decofCoeffs} \color{red}(dashed red line) \color{black} and actual decay of $(c(u)^*_n)^{\frac{1}{2}}$ \color{blue} (solid blue line)\color{black}. For better  \color{black} }\label{fig:1}
\end{figure}

We give a numerical example to illustrate Theorem \ref{thm:1}. We depict in Figure \ref{fig:1} the shearlet coefficients of a function that has a cartoon-like derivative in $x_2$-direction. By \eqref{eq:decofCoeffs} these coefficients should decay as $O(N^{-2})$ for $N\to \infty$. In fact the actual decay matches very closely the predicted decay of $O(N^{-2})$ for $N\to \infty$. The experiments where made with a subsampled version of the shearlet decomposition algorithm of ShearLab, \cite{shearLab}.

\subsection{A different cartoon model}\label{sec:DiffCartoon}

We saw in the proof of Theorem \ref{thm:1} in \eqref{eq:bottleneck} that the main bottleneck for the decay of the shearlet coefficients is \eqref{eq:bottleneck} due to the regularity of the functions $g_1, g_2$ of the cartoon-like function $f = g_1 + \chi_D g_2$. We can examine what happens when we assume more regularity. To keep technicalities at a minimum, we only consider the case where $g_1,g_2 \in C^\infty$, but other cases can be studied similarly.
We study the following cartoon model: Let $1\leq l \in \N$, and 
\begin{equation*}
\mathcal{E}^{l,\infty}(v)  : = \left\{u\in W^{l}_2(\R^2) \ : \ \frac{\partial^\alpha}{\partial x^\alpha} u = u^\alpha \text{ and }u^\alpha \in \mathcal{E}^\infty(v) \text{ for all }  |\alpha| = l\right\},
\end{equation*}
where 
\begin{align*}
\mathcal{E}^\infty(v) = \{f \in L_2(\R^2): f = g_1 + \chi_B g_2, \text{ where } g_1,g_2 \in C^\infty, \suppp g_1,g_2 \subset (0,1)^2\\
\text{ and } B \subset (0,1)^2 \text{ with } \partial B \in C^{\infty} \text{ and }\partial B \text{ has curvature bounded by $v$} \}.
\end{align*}
Obviously $\mathcal{E}^{l,\infty}(v) \subset \mathcal{E}^{l,2}(v)$. From Theorem \ref{thm:0} we know that the best $N$-term approximation rate of wavelets for this class of functions is bounded from below by $N^{-(l+1)}$.

We can now state the approximation rate of shearlet systems for the class $\mathcal{E}^{l,\infty}(v)$. In the proof we will make use of Lemma \ref{lem:decayTheorem} which is given subsequent to the following theorem:

\begin{theorem} \label{thm:2}
Let $l\geq 1$, and $0\leq v<\infty$, and $u\in\mathcal{E}^{l,\infty}(v)$ and let $\mathcal{SH}(\phi,\sh, \tilde{\sh}, c) = (\sh_{\lambda})_{\lambda \in \Lambda}$ be a shearlet system with shearlet generator $\sh  = \left(\frac{\partial}{\partial x_1}\right)^l \theta$ such that with $L\in \N$, $L>l+29/4$, 
$\theta\in C^L(\R^2)$ has compact support, is separable and
\begin{compactenum}[(i)]
 \item $|\hat{\theta}(\xi)| \lesssim \min(1, |\xi_1|^\tau) \min(1,|\xi_1|^{-\nu}) \min(1, |\xi_2|^{-\nu})$ \text{ and }
 \item $|\frac{\partial}{\partial \xi_2}\hat{\theta}(\xi)| \leq |h(\xi_1) (1+\frac{\xi_2}{\xi_1})^{-\nu}|$,
\end{compactenum}
where $\tau>5, \nu\geq 4, h\in L_1(\R)$. Assume that $\theta$ has $M\geq 2L$ vanishing moments in $x_1$-direction. 
Further assume that $(\sh_{\lambda})_{\lambda \in \Lambda}$ forms a frame. Then for every $\epsilon >0$ 
\begin{align}
\|u-\sum_{\lambda \in E_N} \langle u, \sh_\lambda \rangle \sh_\lambda^d \|_{L_2(\R^2)}^2 \lesssim N^{-2l-\frac{7}{2} + \epsilon},
\end{align}
where $E_N \subset \Lambda$ contains the indices of the $N$ largest coefficients $|\langle u, \sh_\lambda \rangle|$ and $\sh_\lambda^d$ is the canonical dual frame element of $\sh_\lambda$. 
\end{theorem}

Before we can present the proof of Theorem \ref{thm:2} we require the following improved estimate for shearlet elements that intersect a discontinuity curve of a cartoon-like function non-tangentially:

\begin{lemma} \label{lem:decayTheorem}
Let $\sh = \sh^1 \otimes \phi^1$ be a separable shearlet where $\sh^1, \phi^1 \in C^L(\R^2) \cap L_2(\R^2)$ are compactly supported and $\sh^1$ has $M\in \N$ vanishing moments.
Let $u = g_1 + \chi_B g_2$ with $g_1,g_2 \in C^P(\R^2)$ and $B \subset (0,1)^2$ with $\gamma = \partial B \in C^{R}$ and $\gamma$ has bounded curvature. Let $R\geq L$ and $L + P\leq M$.

Let $(j,k,m, \iota) \in \Lambda$ such that there exists $t\in \suppp \sh_{j,k,m, \iota}\cap \gamma$ such that the normal $n(t)$ of $\gamma$ at $t$ obeys $n(t) = \pm(\cos \rho_0, \sin \rho_0)$ for some $\rho_0 \in [-\pi/2,\pi/2]$ and $d_{[-\pi/2, \pi/2]}(\arctan((2^{-j/2}k)^\iota), \rho_0) \geq  \beta>0$.
Then we have 
\begin{equation*}
|\left \langle \sh_{j,k,m, \iota}, u\right \rangle| \leq C_u( 2^{-(P+1) \frac{j}{2}} + 2^{-(L+1)\frac{j}{2} + \frac{3}{4}j}),
\end{equation*}
where $C_u$ is a constant depending only on $u$.
\end{lemma}
\begin{proof}
Assume w.l.o.g. that $m = 0$, $k = 0$, $\iota = 1$, for general $|k| \leq 2^{j/2}$ one can apply a transformation to revert back to $k = 0$. The proof for the case $\iota = -1$ is identical to $\iota = 1$. Since $k = 0$, $\iota = 1$ we have that $\rho_0 \not \in (-\beta,\beta)$ and hence locally in a neighborhood of $t = (t_1,t_2)$, there exists $E:\R\to \R$ such that $\gamma$ is given by $x_1\mapsto (x_1+t_1,E(x_1)+t_2)$. $E$ is bounded in a neighborhood of $t$ independent of the chosen $t$ since its slope is bounded due to the constraint on $\rho_0$. Furthermore $E$ is $R$-times continuously differentiable.
Now we can apply the transformation theorem to the above equation to obtain
\begin{align*}
 \left \langle \sh_{j,0,0, 1}, u\right \rangle = \int_{\R^2}\sh_{j,0,0, 1} (x_1, x_2 - E(x_1)) u(x_1, x_2 - E(x_1)) dx. 
\end{align*}
The discontinuity curve of $u$ is locally given by $x \mapsto (x_1+t_1, t_2)$ and hence $u(x_1, x_2 - E(x_1))$ is $C^{P}$ for fixed $x_2$. We can also assume, that for some $c>0$ we have $\suppp \sh_{j,0,0,1} \subset A_j^{-1}([-c,c]^2)$.
Hence we can replace $(x_1, x_2)\mapsto u(x_1, x_2 - E(x_1))$ by a $P$-th order Taylor approximation $U$ to obtain
\begin{equation*}
|\int_{A_j^{-1}([-c,c]^2)}\sh_{j,0,0, 1} (x_1, x_2 - E(x_1)) u(x_1, x_2 - E(x_1)) dx - \int_{A_j^{-1}([-c,c]^2 )}\sh_{j,0,0, 1} (x_1, x_2 - E(x_1)) U(x_1,x_2) dx| \lesssim 2^{-(P+1) \frac{j}{2}}
\end{equation*}
Now we decompose $\sh_{j,0,0, 1}(x_1, x_2 - E(x_1))  = 2^{j/2}\sh^1(2^j x_1 )2^{\frac{j}{4}}\phi^1(2^{\frac{j}{2}} (x_2 -  E(x_1)))$. Let $H(\cdot, x_2)$ be a $L$-th order Taylor approximation of $x_1 \mapsto \phi^1(2^{\frac{j}{2}}(x_2 - E(x_1))) )$, then
\begin{align*}
&\int_{A_j^{-1}([-c,c]^2 )} 2^{\frac{j}{2}}\sh^1(2^j x_1)2^{\frac{j}{4}} \phi^1(2^{\frac{j}{2}} (x_2 - E(x_1))) U(x_1, x_2 ) dx\\ 
=&\int_{A_j^{-1}([-c,c]^2 )} 2^{\frac{j}{2}}\sh^1(2^j x_1 )2^{\frac{j}{4}} H(x_1,x_2) U(x_1, x_2) dx +O(2^{-(L+1)\frac{j}{2} + \frac{3}{4}j}).
\end{align*}
By construction $H U$ is a polynomial of order $L + P$. Since $\sh^1$ has $M\geq L+P$ vanishing moments 
\begin{equation*}
\intRtwo 2^{\frac{j}{2}}\sh^1(2^j x_1 )2^{\frac{j}{4}} H(x_1,x_2) U(x_1, x_2) dx = 0,
\end{equation*}
if $L + P\leq M$.
We obtain 
\begin{equation*}
|\left \langle \sh_{j,0,0, 1}, u\right \rangle| \lesssim 2^{-(P+1) \frac{j}{2}} + 2^{-(L+1)\frac{j}{2} + \frac{3}{4}j}.
\end{equation*}
\end{proof}

\begin{proof}[Proof (of Theorem \ref{thm:2}):]
We use the same notation as in the proof of Theorem \ref{thm:1} for the functions $u_1 = \left(\frac{\partial}{\partial x_1} \right)^l u$ and $u_{-1} = \left(\frac{\partial}{\partial x_2} \right)^l u$ and the discontinuity curves $\gamma_\iota$, $\iota = -1,1$ and the outer normal $n^\iota(t)$ of $\gamma_\iota$ at $t$ for $\iota= {1,-1}$. We decompose our index set $\Lambda$. The sets $\Lambda^0$ and $\Lambda^1$, $\Lambda^{2}$ are defined exactly as in the proof of Theorem \ref{thm:1}. We do, however, decompose $\Lambda^{2}$ differently, into $\Lambda^{2,a}, \Lambda^{2,b}$ and $\Lambda^{2,c}$, to be defined below.

We denote by $\Lambda^{2,a}$ the indices in $\Lambda^{2}: = \Lambda \setminus (\Lambda^0 \cup \Lambda^1)$, such that for $(j,2^{j/2}s,m,\iota) \in \Lambda$ there exists some $t\in \suppp \sh_{j,2^{j/2}s,m,\iota} \cap \gamma_\iota$ such that we have  $s^{\iota}\in (\tan \rho_0-3\cdot2^{-j/2}, \tan \rho_0+3\cdot2^{-j/2})$ and $n^{\iota}(t)= \pm(\cos \rho_0, \sin \rho_0)$ for some $\rho_0 \in (-\pi/2,\pi/2)$. 

Additionally, there are non-aligned shearlet elements, where the angle between the shearlet and the discontinuity curve is less than $\pi/4$ and that touch the discontinuity curve, we will call these indices $\Lambda^{2,b}$. More precisely, these are indices $(j,2^{j/2}s,m,\iota)\in \Lambda$ that are not in $\Lambda^{0}\cup \Lambda^1 \cup \Lambda^{2,a}$ and for which there exists some $t\in \suppp \sh_{j,2^{j/2}s,m,\iota} \cap \gamma_\iota$ such that we have that $d_{[-\pi/2, \pi/2]}(\arctan(s^\iota), \rho_0) <  \pi/4$ and $n^{\iota}(t)= \pm(\cos \rho_0, \sin \rho_0)$ for some $\rho_0 \in (-\pi/2,\pi/2)$.

Lastly, there are indices, such that the angle between the curve and the shearlet is larger than $\pi/4$, i.e. $(j,2^{j/2}s,m,\iota)\in \Lambda \setminus (\Lambda^{0}\cup \Lambda^1 \cup \Lambda^{2,a} \cup \Lambda^{2,b})$ such that $d_{[-\pi, \pi]}(\arctan(s^\iota),\rho_0) \geq  \pi/4$ we will call these indices $\Lambda^{2,c}$. 

For $j\in \Z$ we denote by $\Lambda^1_j,\Lambda^{2,a}_j, \Lambda^{2,b}_j, \Lambda^{2,c}_j$ the indices of the respective index set, with scale equal to $j$. We have that $\#\{(j,k,m,\iota) \in \Lambda^1_j:\langle u, \sh_{j,k,m,\iota}\rangle \neq 0\} \lesssim 2^{2j}$. 
Furthermore,
\begin{equation*}
|\Lambda^{2,a}_j| \lesssim 2^{\frac{j}{2}}, \quad |\Lambda^{2,b}_j| \lesssim 2^{j}, \quad |\Lambda^{2,c}_j| \lesssim 2^{\frac{3}{2}j}.
\end{equation*}

Again we consider three different parts:

\textbf{Part 1: }$\Lambda^0$:

Since $|\Lambda^0|<\infty$ we certainly have
\begin{equation*}
\sum_{n\geq N} c(u)^*_n  \lesssim N^{-2(l+\frac{7}{4})}, 
\end{equation*}
where $c(u)^*$ denotes a non-increasing rearrangement of $(|\langle u, \sh_{j,k,m,\iota} \rangle|^2)_{(j,k,m,\iota)\in \Lambda^0}$.

\textbf{Part 2: }$\suppp \sh_{j,k,m,\iota} \cap \bigcup_{\iota = \{-1,1\}} \gamma_\iota = \emptyset:$

Using \eqref{eq:decayForDeriv} we obtain that 
\begin{align*}
|\langle u, \sh_{j,k,m,\iota} \rangle |\lesssim 2^{-l j} |\langle u_\iota, \theta_{j,k,m,\iota} \rangle|. 
\end{align*}
Since $u_\iota \in C^\infty$ and $\theta$ has $M > l+ 9/2$ vanishing moments in $x_1$-direction and $\suppp \theta_{j,k,m,\iota}$ is of length $2^{-j}$ in the direction indicated by $\iota$ we can estimate
$$
|\langle u, \sh_{j,k,m,\iota} \rangle |\lesssim 2^{-2(l+ \frac{9}{4})j}. 
$$
Let $\epsilon >0$ and $p = (l+9/4-\epsilon)^{-1}$. From the cardinality of $\Lambda^{1}_j$ we have that 
$$
\sum_{\Lambda^{1}} |\langle u, \sh_{j,k,m,\iota}\rangle|^p  \leq \sum_{j\in \N} 2^{2j}  2^{-2j\frac{(l+ \frac{9}{4})}{(l+\frac{9}{4}-\epsilon)}}<\infty.
$$
Using Stechkin's lemma \eqref{eq:Stechkin}, we obtain that $\sum_{n\geq N} c(u)^*_n  \lesssim N^{-2 (l + 7/4 - \epsilon)}$, where $c(u)^*$ denotes a non-increasing rearrangement of $(|\langle u, \sh_{j,k,m,\iota} \rangle|^2)_{(j,k,m,\iota)\in \Lambda^1}$. 

\textbf{Part 3: }$\suppp \sh_{j,k,m,\iota} \cap \bigcup_{\iota = \{-1,1\}} \gamma_\iota \neq \emptyset:$

As already established in the proof of Theorem \ref{thm:1} we have for $(j,k,m,\iota) \in \Lambda^{2,a}$
\begin{equation*}
|\langle u, \sh_{j,k,m,\iota}\rangle| \lesssim 2^{-(\frac{3}{4}+l)j}
\end{equation*}
and for $(j,k,m,\iota) \in \Lambda^{2,b}$ we have
\begin{equation*}
|\langle u, \sh_{j,k,m,\iota}\rangle| \lesssim 2^{-(\frac{9}{4}+l)j}.
\end{equation*}
We continue with the coefficients of $\Lambda^{2,c}$. First of all we invoke \eqref{eq:decayForDeriv} and obtain
\begin{align*}
 |\langle u, \sh_{j,k,m,\iota}\rangle | \lesssim 2^{-l j} |\langle u_\iota, \theta_{j,k,m,\iota} \rangle|. 
\end{align*}
Next we want to estimate $|\langle u_\iota, \theta_{j,k,m,\iota} \rangle|$ using Lemma \ref{lem:decayTheorem}. 
Since $u_\iota\in \mathcal{E}^\infty(v)$ we have that 
$$
u_\iota = g_1 + \chi_B g_2, \text{ with } g_1,g_2 \in C^{P} 
$$
for all $P \in \N$ and in particular for $P = \lceil l + \frac{23}{4}\rceil$. In addition we have that $\partial D \in C^\infty$. Furthermore $\theta \in C^{L}$ has $M \geq 2L$ vanishing moments and $L\geq \lceil l + \frac{29}{4}\rceil$. Hence we can apply Lemma \ref{lem:decayTheorem} to obtain that
\begin{equation*}
|\langle u_\iota, \theta_{j,k,m,\iota} \rangle |\lesssim 2^{-(l+\frac{23}{4} + 1) \frac{j}{2} } .
\end{equation*}
Consequently 
\begin{align*}
 |\langle u, \sh_{j,k,m,\iota}\rangle | \lesssim 2^{-l j} 2^{-(l+\frac{27}{4}) \frac{j}{2} } = 2^{\frac{3}{2}(l + \frac{9}{4})}. 
\end{align*}
Invoking the sizes of $\Lambda^{2,a}_j, \Lambda^{2,b}_j, \Lambda^{2,c}_j$ we can compute, that 
\begin{align}
\sum_{\Lambda^{2}} |\langle u, \sh_{j,k,m,\iota}\rangle|^p \leq& \sum_{\Lambda^{2,a}} |\langle u, \sh_{j,k,m,\iota}\rangle|^p +  \sum_{\Lambda^{2,b}} |\langle u, \sh_{j,k,m,\iota}\rangle|^p+  \sum_{\Lambda^{2,c}} |\langle u, \sh_{j,k,m,\iota}\rangle|^p\nonumber \\
\lesssim &\sum_{j\in \N} 2^{j} 2^{-p(\frac{9}{4}+l)j} +  \sum_{j\in \N} 2^{\frac{j}{2}} 2^{-p(\frac{3}{4}+l)j} +  \sum_{j\in \N} 2^{\frac{3 j}{2}} 2^{-\frac{3}{2} p(l+\frac{9}{4})j}<\infty.
\end{align}
Another application of Stechkin's lemma yields that $\sum_{n\geq N} c(u)^*_n  \lesssim N^{-2 (l + 7/4 - \epsilon)}$, where $c(u)^*$ denotes a non-increasing rearrangement of $(|\langle u, \sh_{j,k,m,\iota} \rangle|^2)_{(j,k,m,\iota)\in \Lambda^2}$. 
Combining the estimates from Part 1, 2, and 3 yields with the dual frame characterization \eqref{eq:frameProp} that
\begin{align*}
\|u-\sum_{\lambda \in E_N} \langle u, \sh_\lambda \rangle \sh_\lambda^d \|_{L_2(\R^2)}^2 \lesssim \sum_{n\geq N} c(u)^*_n  \lesssim N^{-2 (l + \frac{7}{4} - \epsilon)},
 \end{align*}
where $c(u)^*_n$ denotes a non-increasing rearrangement of $(|\langle u, \sh_{\lambda} \rangle|^2)_{\lambda \in \Lambda}$.
\end{proof}

\section{Acknowledgments}
The author would like to thank Wang-Q Lim and Reinhold Schneider for inspiring discussions. The author was supported by the DFG Collaborative Research Center TRR 109 "Discretization in Geometry and Dynamics".

\small
\bibliographystyle{amsplain}
\bibliography{references}

\end{document}